\documentclass{amsart}
\def\0{\{0\}}
\let\>=\rangle
\let\<=\langle
\let\noi=\noindent
\let\sse=\subseteq

\newsymbol\varnothing 203F

\def\span{{\rm span}}
\def\Real{{\rm Re\kern1pt}}
\def\smallfrac#1#2{{\textstyle{\frac{#1}{#2}}}}

\def\B{{\mathcal B}}
\def\H{{\mathcal H}}

\def\N{{\mathcal N}}
\def\R{{\mathcal R}}
\def\X{{\mathcal X}}

\def\CC{{\mathbb C\kern.5pt}}

\def\smallmatrix#1{\null\,\vcenter{
   \baselineskip=8pt\mathsurround=0pt\ialign{
   \hfil ${\scriptstyle##}$
   \hfil &&
   \hfil ${\scriptstyle##}$
   \hfil \crcr
   \mathstrut \crcr
   \noalign{\kern-\baselineskip}#1 \crcr
   \mathstrut \crcr
   \noalign{\kern-\baselineskip} \crcr }}\!}

\newtheorem{theorem}{Theorem}

\newtheorem{corollary}{Corollary}

\theoremstyle{definition}

\newtheorem{remark}{Remark}

\numberwithin{theorem}{section}
\numberwithin{lemma}{section}
\numberwithin{corollary}{section}
\numberwithin{proposition}{section}
\numberwithin{conjecture}{section}
\numberwithin{definition}{section}
\numberwithin{remark}{section}
\numberwithin{question}{section}
\numberwithin{example}{section}
\numberwithin{equation}{section}

\begin{document}

\vglue-55pt
\hfill
{\it Linear and Multilinear Algebra}\/,
{\bf 72}(18) (2024) 3217--3230

\vglue25pt
\title{Forms of Biisometric Operators and Biorthogonality}
\author{B.P. Duggal}
\address{Faculty of Sciences and Mathematics, University of Ni\v s, Serbia}
\email{bpduggal@yahoo.co.uk}
\author{C.S. Kubrusly}
\address{Catholic University of Rio de Janeiro, Brazil}
\email{carlos@ele.puc-rio.br}
\renewcommand{\keywordsname}{Keywords}
\dedicatory{Dedicated to the memory of our friend and colleague Nhan Levan
            $($1936--2021$)$}
\date{May 31, 2023; revised, December 9, 2023}

\begin{abstract}
The paper proves two results involving a pair $(A,B)$ of $P$-biisomet\-ric or
$(m,\!P)$-biisometric Hilbert-space operators for arbitrary positive integer
$m$ and positive operator $P.$ It is shown that if $A$ and $B$ are power
bounded and the pair $(A,B)$ is $(m,P)$-biisometric for some $m$, then it is a
$P$-biisometric pair$.$ The important case when $P$ is invertible is
treated in detail$.$ It is also shown that if $(A,B)$ is $P$-biisometric,
then there are biorthogonal sequences with respect to the inner product
${\<\,\cdot\,;\cdot\,\>_P}={\<P\,\cdot\,;\cdot\,\>}$ that have a shift-like
behaviour with respect to this inner product.
\end{abstract}

\maketitle

\vskip-10pt\noi
\section{Introduction}

Let ${(\H,\<\,\cdot\,;\cdot\,\>)}$ be a Hilbert space and let $A$ and $B$ be
Hilbert-space operators$.$ They are said to make a biisometric pair if
${A^*B=I}$, where $I$ is the identity operator$.$ This extends the notion of
a Hilbert-space isometry$:$ an operator $A$ is an isometry if and only if
${A^*A=I}.$ Biisometric pairs have been investigated in \cite{KL1}$.$ We show
here that this is connected to the notion of $(m,P)$-biisometric pairs and, in
particular, to the notion of $P$-biisometric pairs, where $m$ is a positive
integer and $P$ is a positive operator.

\vskip6pt
The original results in this paper are stated and proved in Sections 4, 5 and
6, viz., Theorems 4.1, 5.1 and 6.1$.$ Let $P$ be a positive operator$.$ In
Theorem 4.1 we show that if $(A,\kern-1ptB)$ is an
$(m,\kern-1ptP)$-biisometric pair for some $m$, and $A$ and $B$ are power
bounded, then $(A,B)$ is a $P$-biisometric pair; so that if in addition $P$ is
invertible, then $A$ and $B$ are similar to a biisometric pair$.$ In
Theorem 5.1 we extend the results in \cite{KL1}, from biisometric pairs to an
arbitrary $P$-biisometric pair, exhibiting a pair
$\big\{\{\phi_n\},\{\psi_n\}\big\}$ of biorthonormal sequences with values in
the inner product space ${(\H,\<\,\cdot\,;\cdot\,\>_P)}$, with inner product
${\<\,\cdot\,;\cdot\,\>_P=\<P\,\cdot\,;\cdot\,\>}$, which have a shift-like
behaviour$.$ Theorem 6.1 gives a comprehensive account of the case when the
positive operator $P$ is invertible (with a bounded inverse) for power bounded
\hbox{operators $A$ and $B$}.

\vskip6pt
All terms and notation used above will be defined in the next section$.$ The
paper is organised as follows$.$ Basic notation and terminology are summarised
in Section 2$.$ Forms of biisometric operators, including $P$-biisometric and
${(m,P)}$-biisometric pairs, are considered in Section 3$.$ Sections 4 and 5
contain the main results of the paper as discussed above$.$ The particular
case when the injective positive $P$ is invertible (with a bounded inverse)
closes the paper in Section 6.

\section{Basic Notation and Terminology}

Throughout this paper ${(\H,\<\,\cdot\,;\cdot\,\>)}$ stands for a complex
Hilbert space, equipped with an inner product ${\<\,\cdot\,;\cdot\,\>}$
generating the norm ${\|\cdot\|}.$ Let ${A\!:\H\to\H}$ be an operator (i.e.,
a bounded linear transformation) of ${(\H,\<\,\cdot\,;\cdot\,\>)}$ into itself
--- referred to as a Hilbert-space operator, or an operator on $\H.$ The
normed algebra of all operators on a normed space $\X$ will denoted by
$\B[\X]$, and so $\B[\B[\X]\kern.5pt]$ stands for the normed algebra of all
bounded linear transformations of $\B[\X]$ into itself (sometimes referred to
as transformers)$.$ Let $I$ stand for the identity operator and $O$ for the
null operator on any linear space$.$ We use the same notation ${\|\cdot\|}$
for the induced uniform norm on $\B[\X]$ for any normed space $\X.$ An
operator $A$ on a normed space $\X$ is power bounded if
${\sup_n\|A^n\|<\infty}$ (which means ${\sup_n\|A^nx\|<\infty}$ for every
${x\in\X}$ if $\X$ is a Banach space by the \hbox{Banach}--Steinhaus
Theorem)$.$ The adjoint of an operator $A$ on any (complex) Hilbert space $\H$
will be denoted by $A^*.\!$ The kernel and range of an operator $A$ on $\H$
will be denoted by $\N(A)$ and $\R(A)$, where $\N(A)$ is a subspace (i.e., a
closed linear manifold) and $\R(A)$ is a linear manifold of
${(\H,\<\,\cdot\,;\cdot\,\>)}$, respectively$.$ Recall that
${\R(A)^-\!=\N(A^*)^\perp}$, where the superscripts $^-$ and $^\perp$ stand
for closure and orthogonal complement in a Hilbert space
${(\H,\<\,\cdot\,;\cdot\,\>)}$, respectively$.$ A Hilbert-space operator $P$
is self-adjoint if ${P^*\!=P}$ (equivalently, if ${\<Px\,;x\>}$ is real for
every ${x\in\H}).$ A self-ad\-joint operator $P$ on Hilbert space is
nonnegative if ${\<Px\,;x\>\ge0}$ for every ${x\in\H}$, and positive if
${\<Px\,;x\>>0}$ for every nonzero ${x\in\H}$ (equivalently, if it is
nonnegative and injective)$.$ An invertible positive operator (with a bounded
inverse; i.e., non\-negative, injective and surjective after the Open Mapping
Theorem) is sometimes referred to as a strictly positive operator$.$ If $P$ is
positive, then (as it is self-adjoint and injective) it is left-invertible
with a dense range (since $\R(P)^-\!=\N(P)^\perp=\0^\perp\!=\H$), so has a
left inverse ${P^{-1}\!:\R(P)\to\H}$, which is bounded if and only if it is
surjective (i.e., if and only it is invertible)$.$ For any positive (i.e.,
injective nonnegative) $P$, the form defined by
$$
\<\,\cdot\,;\cdot\,\>_P=\<P\,\cdot\,;\cdot\,\>\!:\H\times\H\to\CC
$$
is an inner product generating a norm ${\|\cdot\|_P}$, so that
${(\H,\<\,\cdot\,;\cdot\,\>_P)}$ is an inner product space$.$ If $P$ is
strictly positive (i.e., invertible nonnegative), then
${(\H,\<\,\cdot\,;\cdot\,\>_P)}$ is a Hilbert space (i.e., it remains
complete)$.$ Let $\span\,S$ denotes the linear span of an arbitrary set
${S\sse\H}$ and let $\bigvee\!S=(\span\,S)^-$ denotes the closure of
$\span\,S.$

\vskip6pt
Biorthogonal sequences were introduced in the context of basis for separable
\hbox{Banach} spaces \cite[Definition 1.4.1]{S}, \cite[Definition 1.f.1]{LT}
(and so $\H$ is {\it separable}\/ if it is spanned by a sequence)$.$ In
a Hilbert space setting (where dual pairs are identified with inner products
after the Riesz Representation Theorem for Hilbert spaces), the notion of
biorthogonality is defined in terms of the inner product$.$ This is extended
for the case of an inner product space ${(\H,\<\,\cdot\,;\cdot\,\>)}$,
as follows.

\vskip6pt\noi
{\narrower{
Two sequences $\{f_n\}$ and $\{g_n\}$ of vectors in an inner product space
${(\H,\<\,\cdot\,;\cdot\,\>)}$ are said to be {\it biorthogonal}\/ (to each
other) if ${\<f_m;g_n\>}=\delta_{m,n}$ where $\delta$ stands for the Kronecker
delta function$.$ If $\{f_n\}$ is such that there exists a sequence $\{g_n\}$
for which $\{f_n\}$ and $\{g_n\}$ are biorthogonal, then it is said that
$\{f_n\}$ {\it admits a biorthogonal sequence}\/ (symmetrically, $\{g_n\}$
admits a biorthogonal sequence), and $\big\{\{f_n\},\{g_n\}\big\}$ is referred
to as a {\it biorthogonal pair}\/ or a {\it biorthogonal system}\/.
}\vskip0pt}

\vskip6pt
If, in addition, ${\|f_n\|=\|g_n\|=1}$ for all $n$, then
$\big\{\{f_n\},\{g_n\}\big\}$ might be said to be a biorthonormal pair$.$
However, it has been shown in \cite[Corollary 2.1]{KL1} that {\it there is no
distinct pair of biorthonormal sequences}$.$ In other words, if two sequences\/
$\{f_n\}$ and\/ $\{g_n\}$ are biorthogonal and if\/ $\|f_n\|=\|g_n\|=1$ for
all\/ $n$, then\/ ${f_n=g_n}$ for all\/ $n.$ Also, if an arbitrary pair
$\{f_n\}$ and $\{g_n\}$ of biorthogonal sequences is such that $f_n=g_n$ for
all $n$, then we get the usual definition of an orthonormal sequence, although
in general neither $\{f_n\}$ nor $\{g_n\}$ are orthogonal \hbox{(much less
orthonormal) sequences.}

\vskip6pt
A sequence $\{f_n\}$ spanning the whole space $\H$ is sometimes called
{\it total}\/$.$ This means $\bigvee\{f_n\}=\H$ (and $\H$ is separable in this
case)$.$ It was pointed out in \cite{You} that {\it if\/ $\{f_n\}$ admits a
biorthogonal sequence\/ $\{g_n\}$, then\/ $\{g_n\}$ is unique if and only if\/
$\{f_n\}$ is total}\/$.$ We will be concerned with total sequences in the
proof of \hbox{Corollary 5.1}.

\vskip6pt
Replacing ${\<\,\cdot\,;\cdot\,\>}$ with ${\<\,\cdot\,;\cdot\,\>_P}$, for some
invertible (or simply injective) nonnegative operator $P$, we get the
definition of a {\it $P$-biorthogonal pair}\/.

\section{Forms of Biisometric Operators}

Let $A$ be an operator on a Hilbert space $\H$, and let $m$ be a positive
integer$.$ There is a myriad of equivalent definitions for a Hilbert-space
isometry$.$ The one that fits our needs here reads as follows: an operator
$A$ is an {\it isometry}\/ if
$$
A^*A=I
\qquad ({\rm i.e.,} \;A^*A-I=O).
$$
By replacing ``$=$'' with ``$\le$'' we get another expression defining a
contraction$.$ Perhaps the above displayed form for an isometry has been
popularised in \cite{Hal}$.$ It seems that the notion of $m$-isometry appeared
in the last decade of the past century \cite{AS}, and a considerable number of
research papers dealing with several aspects of it has been noticed recently
(see, e.g., \cite{Dug1,Suc2})$.$ An operator $A$ is an {\it $m$-isometry}\/ if

$$
{\sum}_{j=0}^m(-1)^j\Big(\smallmatrix{m \cr
                                      j \cr}\Big)A^{*(m-j)}A^{m-j}=O
$$
for some positive integer $m.$ A $1$-isometry is precisely a plain isometry$.$
On the other hand, but still along the same line, there also is the notion of
$P$-isometry with respect to an injective nonnegative operator $P$: An
operator $A$ is a {\it $P$-isometry}\/ if
$$
A^*PA=P
\qquad ({\rm i.e.,} \;A^*PA-P=O),
$$
reducing to a plain isometry if $P$ is the identity operator$.$ (For recent
papers dealing with $P$-isometry --- and its variations as, for instance,
$P$-contractions, see, e.g., \cite{Suc1,KD3})$.$ The above two notions prompt
the next one$.$ An operator $A$ is an {\it $(m,P)$-isometry}\/ (for a positive
integer $m$ and a positive operator $P$) if
$$
{\sum}_{j=0}^m(-1)^j\Big(\smallmatrix{m \cr
                                      j \cr}\Big)A^{*(m-j)}PA^{m-j}=O.
$$
Again, a $(1,P)$-isometry is $P$-isometry$.$ (For $(m,P)$-isometries and their
variations, such as $(m,P)$-expansive operators where the ``$=$'' sign is
replaced by ``$\le$'', see, e.g., \cite{DKim1,Tri})$.$ The above notions are
extended to a pair of operators as follows$.$ Let $B$ be another operator on
$\H.$ Operators $A$ and $B$ are said to make a \hbox{{\it biisometric pair}\/
\cite{KL1} if}
$$
A^*B=I 
\qquad ({\rm i.e.,} \;A^*B-I=O).
$$
Equivalently, if ${B^*A=I}.$ The pair $(A,B)$ is said to be an 
{\it $m$-biisometric pair}\/ if
$$
{\sum}_{j=0}^m(-1)^j\Big(\smallmatrix{m \cr
                                      j \cr}\Big)A^{*(m-j)}B^{m-j}=O
$$
for a positive integer $m$ (see, e.g., \cite{DKim2})$.$ We say that $A$ and
$B$ make a {\it $P$-biisometric pair}\/ if, for a positive operator $P$,
$$
A^*PB=P 
\qquad ({\rm i.e.,} \;A^*PB-P=O).
$$
The above two expressions naturally lead to the notion of
{\it $(m,P)$-biisometric pair}\/:
$$
{\sum}_{j=0}^m(-1)^j\Big(\smallmatrix{m \cr
                                      j \cr}\Big)A^{*(m-j)}PB^{m-j}=O,
$$
where an $(m,I)$-biisometric is $m$-biisometric and a $(1,P)$-biisometric is
$P$-biisomet\-ric$.$ The present paper focuses on the last two notions.

\section{On $(m,P)$-Biisometric and $P$-Biisometric Pairs}

Given operators $A,B$ in $\B[\H]$, let ${L_A,R_B\in\B[\B[\H]]}$ denote the
operators of left multiplication by $A$ and, respectively, right
multiplication by $B$, given by
$$
L_A(X)=AX
\quad\hbox{and}\quad
R_B(X)=XB
\quad\hbox{for every}\quad
X\in\B[\H].
$$
Then set
\begin{eqnarray*}
\triangle^m_{A^*,B}(P)
&=&
(L_{A^*}R_B-I)^m(P)                                                     \\
&=&
\Big({\sum}_{j=0}^m(-1)^j\Big(\smallmatrix{m \cr
                                           j \cr}\Big)
(L_{A^*}R_B)^{m-j}\Big)(P)                                              \\
&=&
{\sum}_{j=0}^m(-1)^j\Big(\smallmatrix{m \cr
                                      j \cr}\Big)A^{*(m-j)}PB^{m-j},
\end{eqnarray*}
so that the pair $(A,B)$ is $(m,P)$-biisometric if and only if
$\triangle^m_{A^*,B}(P)=0.$ Since 
$$
\triangle^m_{A^*,B}(P)
=(L_{A^*}R_B-I)\big(\triangle^{m-1}_{A^*,B}(P)\big)
=L_{A^*}R_B\big(\triangle^{m-1}_{A^*,B}(P)\big)-\triangle^{m-1}_{A^*,B}(P),
$$
if a pair $(A,B)$ is $(m,P)$-biisometric, then
\begin{eqnarray*}
&&
\triangle^{m-1}_{A^*,B}(P)=A^*\triangle^{m-1}_{A^*,B}(P)B                \\
&\Longrightarrow&
\triangle^{m-1}_{A^*,B}(P)
=A^*\triangle^{m-1}_{A^*,B}(P)B
={A}^{*2}\triangle^{m-1}_{A^*,B}(P)B^2                                   \\
&\Longrightarrow&
\triangle^{m-1}_{A^*,B}(P)
=A^*\triangle^{m-1}_{A^*,B}(P)B
=\cdots
=A^{*n}\triangle^{m-1}_{A*,B}(P)B^n
\end{eqnarray*}
for all positive integers $n.$ In particular,
$$
\hbox{\it a $P$-biisometric pair $(A,B)$ is $(m,P)$-biisometric for positive
integers $m$}.
$$
Does the converse hold$?$ The answer given here, in the absence of any
additional hypotheses, is an emphatic ``no''$:$ for example, if
$A=B=\big(\smallmatrix{1 & 1 \cr
                       0 & 1 \cr}\big)$ and $P$ is the positive operator
$P=\big(\smallmatrix{1 & 1 \cr
                     1 & 2\cr}\big)$, then $A$ is a $(3,P)$-isometry but not
a $P$-isometry (i.e., the pair $(A,A)$ is $(3,P)$-isometric but not
$P$-isometric)$.$ The following \hbox{theorem says that}
\vskip6pt\noi
{\narrower\narrower
{\it a necessary and sufficient condition for an $(m,P)$-biisometric pair
$(A,B)$ to be $P$-biisometric, for a positive ${P\in B[\H]}$, is that
$L_{A^*}R_B$, given by $(L_{A^*}R_B)(P)=A^*PB$, is power bounded.}
\vskip6pt}
\noi

\begin{remark}
\label{remark}
Take arbitrary operators ${A,B}$, let the $L_A$ be the left multiplication
by A, and $R_B$ be the right multiplication by $B$, as defined above$.$ The
commutativity of $L_A$ and $R_B$ ensures that
$$
(L_AR_B)^n=L_A^nR_B^n=L_{A^n}R_{B^n},
$$
\vskip-4pt\noi
and so
\vskip4pt\noi
$$
(L_{A}R_B)^n(X)
=(L_{A}^nR_B^n)(X)
=(L_{A^n}R_{B^n})(X)
=A^nXB^n
$$
for every nonnegative integer $n$ and every operator $X$. Since 
\begin{eqnarray*}
\triangle^m_{A^n,B^n}(X)
&=&(L_{A^n}R_{B^n}-I)^m(X)
=(L^n_AR^n_A-I)^m(X)                                                        \\
&=&
\Big[(L_AR_B-I)^m{\sum}_{j=0}^{m(n-1)}{\alpha_j}(L_AR_B)^{m(n-1)-j}\Big](X) \\
&=&
{\sum}_{j=0}^{m(n-1)}{\alpha_j}(L_AR_B)^{m(n-1)-j}(\triangle^m_{A,B}(X))
\end{eqnarray*}
for some scalars $\alpha_j$, if $(A^*,B)$ is $(m,X)$-biisometric, then
$(A^{*n},B^n)$ is $(m,X)$-biiso\-metric for all positive integers $n$. Again,
since
$$
\|(L_AR_B)^n\|
={\sup}_{X\ne O}\smallfrac{\|(L_AR_B)^n(X)\|}{\|X\|}
={\sup}_{X\ne O}\smallfrac{\|A^n\kern-1ptXB^n\|}{\|X\|}
\le{\sup}_n\|A^n\|\,{\sup}_n\|B^n\|,
$$
$L_AR_B$ is power bounded whenever both $A$ and $B$ are power bounded$.$ Since
an operator $A$ is power bounded if and only if its adjoint $A^*$ is power
bounded, power boundedness of $A$ and $B$ implies power boundedness of
$L_{A^*}R_B$, which is the condition of the next theorem.
\end{remark}

\begin{theorem}
\label{theorem0}
Let\/ $A$ and\/ $B$ be operators on a Hilbert space\/
$(\H,\<\,\cdot\,;\cdot\,\>)$ such that the pair\/ $(A,B)$ is\/
$(m,P)$-biisometric for some positive integer\/ $m$ and positive operator\/
$P.$ If\/ $L_{A^*}R_B$ is power bounded, then the pair\/ $(A,B)$ is\/
$P$-biisometric.
\end{theorem}

\proof
The easily proved (use an induction argument) identity
$$
(a-1)^t=a^t-{\sum}_{j=0}^{t-1}\Big(\smallmatrix{t \cr
                                                j \cr}\Big)(a-1)^j
$$
\vskip-2pt\noi
for all positive integers $t$ implies
\begin{eqnarray*}
\triangle^m_{A^*,B}(P)
&=&
(L_{A^*}R_B-I)^m(P)                                                      \\
&=&
(L_{A^*}R_B)^m(P)-{\sum}_{j=0}^{m-1}\Big(\smallmatrix{m \cr
                                                      j \cr}\Big)
\triangle^j_{A^*,B}(P)                                                   \\
&=&
A^{*m}PB^m-{\sum}_{j=0}^{m-1}\Big(\smallmatrix{m \cr
                                               j \cr}\Big)
\triangle^j_{A^*,B}(P).
\end{eqnarray*}
Hence, if $(A,B)$ is $(m,P)$-biisometric, then
\begin{eqnarray*}
O
&=&
A^{*m}PB^m-{\sum}_{j=0}^{m-1}\Big(\smallmatrix{m \cr
                                               j \cr}\Big)
\triangle^j_{A^*,B}(P)                                                   \\
\Longrightarrow\quad
O
&=&
A^{*(m+1)}PB^{m+1}-{\sum}_{j=0}^{m-1}\Big(\smallmatrix{m \cr
                                                       j \cr}\Big)
A^*\triangle^j_{A^*,B}(P)B                                               \\
&=&
A^{*(m+1)}PB^{m+1}-{\sum}_{j=0}^{m-1}\Big(\smallmatrix{m \cr
                                                       j \cr}\Big)
\triangle^{j+1}_{A^*,B}(P)-{\sum}_{j=0}^{m-1}\Big(\smallmatrix{m \cr
                                                               j \cr}\Big)
\triangle^j_{A^*,B}(P)                                                   \\
&=&
A^{*(m+1)}PB^{m+1}-\Big(\smallmatrix{m   \cr
                                     m-1 \cr}\Big)
\triangle^{m}_{A^*,B}(P)-{\sum}_{j=0}^{m-1}\Big(\smallmatrix{m+1 \cr
                                                             j   \cr}\Big)
\triangle^j_{A^*,B}(P)                                                   \\
&=&
A^{*(m+1)}PB^{m+1}-{\sum}_{j=0}^{m-1}\Big(\smallmatrix{m+1 \cr
                                                       j   \cr}\Big)
\triangle^j_{A^*,B}(P).
\end{eqnarray*}
An induction argument now leads us to the conclusion that
\begin{eqnarray*}
O
&=&
A^{*n}PB^{n}-{\sum}_{j=0}^{m-1}\Big(\smallmatrix{n \cr
                                                 j \cr}\Big)
\triangle^j_{A^*,B}(P)                                                   \\
&=&
A^{*n}PB^n-\Big(\smallmatrix{n   \cr
                             m-1 \cr}\Big)\triangle^{m-1}_{A^*,B}(P)
-{\sum}_{j=0}^{m-2}\Big(\smallmatrix{n \cr
                                     j \cr}\Big)\triangle^j_{A^*,B}(P)   \\
\Longleftrightarrow\quad
A^{*n}PB^n
&=&
\Big(\smallmatrix{n   \cr
                  m-1 \cr}\Big)\triangle^{m-1}_{A^*,B}(P)
+{\sum}_{j=0}^{m-2}\Big(\smallmatrix{n \cr
                                     j \cr}\Big)
\triangle^j_{A^*,B}(P)
\end{eqnarray*}
for all integers $n\geq m.$ Assume that $L_{A^*}R_B$ is power bounded, which
means that there exists a positive scalar $M$ such that
${\|L_{A^{*n}}R_{B^n}\|\le M}$ for all $n.$ (Recall that the power boundedness
of ${L_{A^*}R_B}$ is guaranteed by the power boundedness of the operators $A$
and $B$.) Then
$$
\limsup_{n\to\infty}{\|A^{*n}PB^n}x\|\leq M\|P\|\,\|x\|
$$
and there exists a positive scalar $M_j$, dependent on $j$, such that
$$
\limsup_{n\to\infty}{\|\triangle^j_{A^*,B}(P)}x\|\leq M_j\|P\|\,\|x\|
$$
for all $x\in\H$. Observe that $\Big(\smallmatrix{n   \cr
                                                  m-1 \cr}\Big)$ is of the
order of $n^{m-1}$ and $\Big(\smallmatrix{n \cr
                                          j \cr}\Big)$, $\,{0\leq j\leq m-2}$,
is of the order of $n^{m-2}$; hence, letting ${n\to\infty}$ in
$$
\|{\triangle^{m-1}_{A^*,B}(P)}x\|\leq
{\frac{1}{\big(\smallmatrix{n   \cr
                            m-1 \cr}\big)}}
\left(\|A^{*n}PB^nx\|+\kern-1pt{\sum}_{j=0}^{m-2}
\Big(\smallmatrix{n \cr
                  j \cr}\Big)\|{\triangle^j_{A^*,B}(P)}x\|\right)
$$
we have
$$
\|{\triangle^{m-1}_{A^*,B}(P)}x\|=0
\;\;\hbox{for all}\;\;
x\in\H
\quad\Longleftrightarrow\quad
\triangle^{m-1}_{A^*,B}(P)=0.
$$
Repeating the argument a finite number of times, this implies
$$
\triangle_{A^*,B}(P)=0
\quad\Longleftrightarrow\quad
A^*PB=P.                                                       \eqno{\qed}
$$

\vskip0pt
$\!$Choosing ${A=B}$ in Theorem 4.1, with $P$ positive (i.e., if $A$ is a
$P$-isometry), we have ${A^*PA=(A^*P^{\frac{1}{2}})(P^{\frac{1}{2}}A)=P}$ and
$A$ is an isometry with respect to the norm ${\|\cdot\|_P}$ (i.e.,
${\|Ax\|_P=\|x\|_p}$ for every ${x\in\H}$ --- see, e.g.,
\cite[Proposition 4.1(b)]{KD3})$.$ So there exists an isometry $V$ such that
${A^*P^{\frac{1}{2}}=P^{\frac{1}{2}}V^*}.$ In other words, the operator $A$
is a $P$-isometry in the sense that
${\|P^{\frac{1}{2}}Ax\|=\|P^{\frac{1}{2}}x\|}$ for all ${x\in\H}.$ The
$P$-isometric property of $A$ for the case of a nonnegative $P$ does not imply
the left invertibility of $A.$ For example, if
$A=\Big(\smallmatrix{0 & 1 & 0 \cr
                     1 & 0 & 0 \cr
                     0 & 0 & 0 \cr}\Big)$ and $P$ is the nonnegative operator
$P=\Big(\smallmatrix{1 & 1 & 0 \cr
                     1 & 1 & 0 \cr
                     0 & 0 & 0 \cr}\Big)$, then $0$ is in the point spectrum
of $A$ and ${A^*PA=P}.$ Such a situation can not, however, arise if $P$ is a
positive (i.e., nonnegative and injective) operator$.$ For, in this case, if
${\{x_n\}\subset\H}$ is a sequence of unit vectors such that
${\lim_{n\rightarrow\infty}{\|({A-\lambda})x_n\|}=0}$ and $A$ is a $P$-isometry
(i.e., the pair $(A,A)$ is $(m,P)$-isometric), then
$$
\lim_{n\to\infty}{\<\triangle^m_{A^*,A}(P)x_n\,;x_n\>}
=(|\lambda|^2-1)^m\lim_{n\to\infty}{\langle Px_n\,;x_n \rangle}=0
$$
implies $(|\lambda|^2-1)=0$; that is, the approximate point spectrum of $A$
lies in the unit circle (hence, $A$ is left invertible).

\section{$P$-Biisometric Operators and Biorthogonal Sequences}

A $P$-biisometric pair of operators was defined in Section 3 and
biorthogonal sequences were defined in Section 2.

\begin{theorem}
\label{theorem1}
Let\/ $A\!$ and\/ $B$ be operators on\/ a Hilbert space\/
${(\H,\<\,\cdot\,;\cdot\,\>)}.$ Suppose they make a\/ $P$-biisometric pair
whose adjoints are noninjective,
$$
\N(A^*)\ne\0
\quad\;\hbox{and}\;\quad
\N(B^*)\ne\0,                                                 \leqno{\rm(i)}
$$
and there exists an injective nonnegative\/ $($i.e., positive\/$)$ operator\/
$P$\/ on\/ $\H$ for which
$$ 
A^*PB=P,
\qquad
\N(A^*)\cap\R(P)\ne\0,
\quad\hbox{and}\quad
\N(B^*)\cap\R(P)\ne\0.                                       
                                                              \leqno{\rm(ii)}
$$
Consider the inner product space\/ ${(\H,\<\,\cdot\,;\cdot\,\>_P)}$ with inner
product given by\/ ${\<\,\cdot\,;\cdot\,\>_P}={\<P\,\cdot\,;\cdot\,\>}.$ Take
arbitrary nonzero vectors
$$
v\in\N(A^*)\cap\R(P)
\quad\;\hbox{and}\;\quad
w\in\N(B^*)\cap\R(P)
$$
and, for each nonnegative integer $n$, consider the vectors
$$
\phi_n=A^nP^{-1}w
\quad\;\hbox{and}\;\quad
\psi_n=B^nP^{-1}v
$$
in\/ $\H.$ We claim that there exist\/ ${v\in\N(A^*)\cap\R(P)}$ and\/
${w\in\N(B^*)\cap\R(P)}$ such that the sequences\/ $\{\phi_n\}$ and\/
$\{\psi_n\}$ are biorthogonal on\/ ${(\H,\<\,\cdot\,;\cdot\,\>_P)}.$ Moreover,
$$
A\phi_n=\phi_{n+1}
\quad\;\hbox{and}\;\quad
B\psi_n=\psi_{n+1},
$$
$$
A^*(P\psi_{n+1})=P\psi_n
\quad\;\hbox{and}\;\quad
B^*(P\phi_{n+1})=P\phi_n.
$$
\end{theorem}

\begin{proof}
Let ${\<\,\cdot\,;\cdot\,\>}$ be an inner product on $\H$, where
${(\H,\<\,\cdot\,;\cdot\,\>)}$ is a Hilbert space$.$ First suppose $A^*$ and
$B^*$ are noninjective (i.e., ${\N(A^*)\ne\0}$ and ${\N(B^*)\ne\0}$, which
is equivalent to saying that $A$ and $B$ have nondense ranges; see, e.g.,
\cite[Propositions 5.12 and 5.76]{EOT})$.$ Next suppose there exists an
injective nonnegative (thus self-adjoint) operator $P$ on $\H$ for which
$$
A^*PB=P
\qquad\;\hbox{(equivalently, $B^*PA=P$)},
$$
so that a trivial induction shows that, for every nonnegative integer
$n$,
$$
A^{*n}PB^n=P
\qquad\;\hbox{(equivalently, $B^{*n}PA^n=P$)}.
$$
Moreover, suppose
$$
\N(A^*)\cap\R(P)\ne\0
\quad\;\hbox{and}\;\quad
\N(B^*)\cap\R(P)\ne\0.
$$
As $P$ is an injective nonnegative operator, the form
${\<\,\cdot\,;\cdot\,\>_P\!:\H\times\H\to\CC}$ given by
$$
\<\,\cdot\,;\cdot\,\>_P=\<P\,\cdot\,;\cdot\,\>
$$
is another inner product in $\H.$ Take arbitrary nonzero vectors
$$
v\in\N(A^*)\cap\R(P)
\quad\;\hbox{and}\;\quad
w\in\N(B^*)\cap\R(P).
$$
Let ${P^{-1}\!:\R(P)\to\H}$ be the left inverse of $P$ and set
$$
y=P^{-1}v\ne0
\quad\;\hbox{and}\;\quad
z=P^{-1}w\ne0
$$
in $\H.$ Now take $\phi_n$ and $\psi_n$ in $\H$ defined for every nonnegative
integer $n$ by
$$
\phi_n=A^nz=A^nP^{-1}w
\quad\;\hbox{and}\;\quad
\psi_n=B^ny=B^nP^{-1}v.
$$
Another trivial induction shows that
$$
A\,\phi_n=A^{n+1}z=\phi_{n+1}
\quad\;\hbox{and}\;\quad
B\psi_n=B^{n+1}y=\psi_{n+1}.
$$
Therefore, since ${A^*PB=P}$, we also get
$$
A^*(P\psi_{n+1})
=A^*PB\psi_n=P\psi_n
\quad\;\hbox{and}\;\quad
B^*(P\phi_{n+1})=B^*PA\,\phi_n=P\phi_n.
$$
Now let ${m,n}$ be a pair of nonnegative integers$.$ If ${m<n}$, then
$$
\!\<\phi_m\,;\psi_n\>_P
\!=\!\<PA^m\kern-.5ptz\,;B^ny\>
\!=\!\<z\,;A^{*m}PB^{m}B^{n-m}\kern-.5pty\>
\!=\!\<Pz\,;B^{n-m}\kern-.5pty\>
\!=\!\<B^{*(n-m)}\kern-.5ptw\,;y\>
$$
and so ${\<\phi_m\,;\psi_n\>_P=0}$ since ${w\in\N(B^*)}$ implies
${w\in\N(B^{*(m-n)})}.$ Symmetrically,
$$
\<\phi_m\,;\psi_n\>_P
=\<B^{*n}PA^nA^{m-n}z\,;y\>
=\<PA^{m-n}z\,;y\>
=\<z\,;A^{*(m-n)}v\>
=0
$$
if ${n<m}$, since ${v\in\N(A^*)}.$ Moreover, for ${m=n}$,
$$
\<\phi_n\,;\psi_n\>_P
=\<PA^nz\,;B^ny\>
=\<z\,;A^{*n}PB^ny\>
=\<z\,;Py\>
=\<z\,;v\>
=\<P^{-1}w\,;v\>.
$$
Summing up.
$$
\<\phi_m\,;\psi_n\>_P=0
\;\;\hbox{whenever}\;\;
m\ne n
\quad\;\hbox{and}\;\quad
\<\phi_n\,;\psi_n\>_P=\<P^{-1}w\,;v\>.
$$
Next we proceed to show that there are ${v\in\N(A^*)\cap\R(P)}$ and
${w\in\N(B^*)\cap\R(P)}$ such that ${\<P^{-1}w\,;v\>\ne0}$, and so (as
${\N(A^*)\cap\R(P)}$ and ${\N(B^*)\cap\R(P)}$ are linear spaces), there exist
a pair ${(v,w)}$ with ${v\in\N(A^*)\cap\R(P)}$ and ${w\in\N(B^*)\cap\R(P)}$
such that ${\<P^{-1}w\,;v\>=1}.$ In this case (that is, for such a pair
${(v,w)}$ of vectors in ${\N(A^*)\cap\R(P)\times\N(B^*)\cap\R(P)}\kern.5pt$)
we get
$$
\<\phi_m\,;\psi_n\>_P=0
\;\;\hbox{if}\;\;
m\ne n
\quad\;\hbox{and}\;\quad
\<\phi_n\,;\psi_n\>_P=1,
$$
so that $\{\phi_n\}$ and $\{\psi_n\}$ are biorthogonal sequences on
${(\H,\<\,\cdot\,;\cdot\,\>_P)}.$ That there is such a pair ${(v,w)}$ for
which ${\<P^{-1}w\,;v\>\ne0}$ is a consequence of the following result.

\vskip9pt\noi
{\it Claim}\/.\hskip64pt
$\N(A^*)\cap\R(P)\not\perp P^{-1}(\N(B^*)\cap\R(P))$.

\vskip9pt\noi
{\it Proof of Claim}\/.
$\kern-1pt$Since ${A\kern-1pt\ne\kern-1ptO}$ (as
${A^*PB\kern-1pt=\kern-1ptP\!\ne\kern-1ptO}$), take
${0\kern-1pt\ne\kern-1ptu\in\kern-1pt\R(A)}$ so that ${u=\kern-1ptAx}$ for
some ${0\ne x\in\H}.$ Suppose ${u\in P^{-1}(\N(B^*)\cap\R(P))}.$ Then
$u=P^{-1}w$ for some ${w\in\R(P)}$ for which $B^*w=0.$ Thus $Ax=u=P^{-1}w$,
and so $PAx=PP^{-1}w=w$ (as ${w\in\R(P)}\kern.5pt).$ Then
$Px=P^*x=B^*PAx=B^*w=0$, so that ${x=0}$ (as ${\N(P)=\0}$), which is a
contradiction$.$ Hence, ${\R(A)\cap P^{-1}(\N(B^*)\cap\R(P))}=\0$, and so
${\R(A)^-\!\cap P^{-1}(\N(B^*)\cap\R(P))}=\0.$ Equivalently (as
${\R(A)^-\!=\N(A^*)^\perp}$),
$$
{\N(A^*)^\perp\cap P^{-1}(\N(B^*)\cap\R(P))}=\0.
$$
Suppose ${P^{-1}(\N(B^*)\cap\R(P))\perp\!\N(A^*)}.$ Then
${P^{-1}(\N(B^*)\cap\R(P))\sse\N(A^*)^\perp}\!.$ So
\begin{eqnarray*}
P^{-1}(\N(B^*)\cap\R(P))
&\kern-6pt=\kern-6pt&
P^{-1}(\N(B^*)\cap\R(P))\cap P^{-1}(\N(B^*)\cap\R(P))                  \\
&\kern-6pt\sse\kern-6pt&
\N(A^*)^\perp\cap P^{-1}(\N(B^*)\cap\R(P))=\0
\end{eqnarray*}
by the above identity, so that ${\N(B^*)\cap\R(P)=\0}$, which is a
contradiction$.$ Hence
$$
\N(A^*)\not\perp P^{-1}(\N(B^*)\cap\R(P)).
$$
Thus there exist ${v'\in\N(A^*)}$ and ${u\in P^{-1}(\N(B^*)\cap\R(P))}$ such
that ${\<v'\,;u\>\ne0}.$ Since $\R(P)$ is dense in $\H$, and since the inner
product is continuous, there exists a vector ${v\in\N(A^*)\cap\R(P)}$ such
that ${\<v\,;u\>\ne0}.$ Therefore
$$
{\N(A^*)\cap\R(P)\not\perp P^{-1}(\N(B^*)\cap\R(P))}.\qed
$$

\vskip2pt\noi
The above claim ensures that there are ${w\in\N(B^*)\cap\R(P)}$ and
${v\in\N(A^*)\cap\R(P)}$ such that ${\<P^{-1}w\,;v\>\ne0}$, concluding the
proof of the theorem.
\end{proof}

\vskip0pt
If the positive $P$ is surjective (i.e., if it is invertible), then the fact
that ${\R(P)=\H}$ simplifies condition (ii)$.$ The above theorem generalises
\cite[Theorem 3.1]{KL1} (also \cite[Theorem 3.1]{KL2})$.$ Indeed, by
setting $P=I$ we get the result in \cite{KL1} as a \hbox{particular case}.

\vskip6pt
The next corollary shows that a $P$-biisometric pair has a shift-like property
regarding biorthogonal sequences with respect to the inner product
${\<\,\cdot\,;\cdot\,\>_P}$ (i.e., with respect to $P$-biorthogonal sequences).

\begin{corollary}
\label{corollary1}
Let\/ $A$ and\/ $B$ be operators on a Hilbert space\/ $\H$ for which there
exists an injective nonnegative operator\/ $P$ on\ $\H$ such that
$$
A^*PB=P,
$$
and consider the biorthogonal sequences\/ $\{\phi_n\}$ and\/ $\{\psi_n\}$
defined in Theorem 5.1 in terms of nonzero vectors\/ ${v\in\N(A^*)\cap \R(P)}$
and\/ ${w\in\N(B^*)\cap \R(P)}.$ In addition, suppose these biorthogonal
sequences span\/ $\H$, and also suppose a vector\/ ${x\in\H}$ has series
expansion in terms of\/ $\{\phi_n\}$ and\/ $\{\psi_n\}.$ Then\/ this\/ $x$ can
be expressed as
$$
x={\sum}_{k=0}^{\infty}\<x\,;\psi_k\>_P\,\phi_k
={\sum}_{k=0}^{\infty}\<x\,;\phi_k\>_P\,\psi_k,
$$
and its image with respect to\/ $A$, $A^*$, $B$, and\/ $B^*$ can be
expressed as
$$
Ax={\sum}_{k=0}^{\infty}\<x\,;\psi_k\>_P\,\phi_{k+1}
\quad\;\hbox{and}\;\quad
Bx={\sum}_{k=0}^{\infty}\<x\,;\phi_k\>_P\,\psi_{k+1},
$$
$$
A^*x={\sum}_{k=0}^{\infty}\<x\,;\phi_{k+1}\>_P\,\psi_k
\quad\;\hbox{and}\;\quad
B^*x={\sum}_{k=0}^{\infty}\<x\,;\psi_{k+1}\>_P\,\phi_k.
$$
\end{corollary}

\proof
Take an arbitrary ${x\in\H}.$ To begin with, note that if a sequence in a
biorthogonal pair spans $\H$, then it does not necessarily follow that all
elements in $\H$ \hbox{have an} expansion as the limit of a linear combination
of elements of the sequence (cf$.$ \cite[Example 5.4.6]{Chr})$.$ Thus first
suppose the biorthogonal sequences $\{\phi_n\}$ and $\{\psi_n\}$ span $\H$
(i.e., $\bigvee\{\phi_n\}=\bigvee\{\psi_n\}=\H$)$.$ According to
\cite[p.537]{You}, $\{\phi_n\}$ admits a \hbox{biorthogonal} sequence if and
only if none of its elements is the limit of a linear combination of the
others, and in this case the biorthogonal sequence $\{\psi_n\}$ is uniquely
determined if and only $\bigvee\{\phi_n\}=\H.$ Thus assume first that the
biorthogonal pair $\big\{\{\phi_n\},\{\psi_n\}\big\}$ is unique in the above
sense$.$ In addition, also suppose that
$$
x={\sum}_{k=0}^{\infty}\alpha_k\phi_k
={\sum}_{k=0}^{\infty}\beta_k\psi_k
$$
for some pair of sequences of scalars $\{\alpha_n\}$ and $\{\beta_n\}.$ Then
$$
\alpha_n=\<x\,;\psi_n\>_P
\quad\;\hbox{and}\;\quad
\beta_n=\<x\,;\phi_n\>_P
$$
for every ${n\ge0}.$ Indeed, by the continuity of the inner product, and
recalling from Theorem 5.1 that $\{\phi_k\}$ and $\{\psi_k\}$ are biorthogonal
with respect to the inner product ${\<\,\cdot\,;\cdot\,\>_P}$, we get
$$
\<x\,;\psi_n\>_P={\sum}_{k=0}^{\infty}\alpha_k\<\phi_k\,;\psi_n\>_P=\alpha_n
\quad\,\hbox{and}\,\quad
\<x\,;\phi_n\>_P={\sum}_{k=0}^{\infty}\beta_k\<\psi_k\,;\phi_n\>_P=\beta_n
$$
for every ${n\ge0}.$ Recall again from Theorem 5.1 that for each ${n\ge0}$
$$
A\,\phi_n=\phi_{n+1}
\quad\;\hbox{and}\;\quad
B\psi_n=\psi_{n+1}.
$$
Using the above identities only (and the continuity of the inner product),
apply $A$ and $B$ to the expansions of $x$ in terms of $\{\phi_n\}$ and
$\{\psi_n\}$, respectively, and apply $A^*$ and $B^*$ to the expansions of $x$
in terms of $\{\psi_n\}$ and $\{\phi_n\}$, respectively, to get
$$
Ax={\sum}_{k=0}^{\infty}\<x\,;\psi_k\>_P\,A\,\phi_k
={\sum}_{k=0}^{\infty}\<x\,;\psi_k\>_P\,\phi_{k+1},
$$
$$
Bx={\sum}_{k=0}^{\infty}\<x\,;\phi_k\>_P\,B\psi_k
={\sum}_{k=0}^{\infty}\<x\,;\phi_k\>_P\,\psi_{k+1},
$$
$$
A^*x={\sum}_{k=0}^{\infty}\<A^*x\,;\phi_k\>_P\,\psi_k
={\sum}_{k=0}^{\infty}\<x\,;A\,\phi_k\>_P\,\psi_k
={\sum}_{k=0}^{\infty}\<x\,;\phi_{k+1}\>_P\,\psi_k,
$$
$$
B^*x={\sum}_{k=0}^{\infty}\<B^*x\,;\psi_k\>_P\,\phi_k
={\sum}_{k=0}^{\infty}\<x\,;B\psi_k\>_P\,\phi_k
={\sum}_{k=0}^{\infty}\<x\,;\psi_{k+1}\>_P\,\phi_k.           \eqno{\qed}
$$

\section{$(m,P)$-Biisometric Pair for a Strictly Positive $P$}

Throughout the paper the operator ${P\in\B[\H]}$ has been assumed positive
(i.e., non\-negative and injective) so that
${\<\,\cdot\,;\cdot\,\>_P}={\<P\,\cdot\,;\cdot\,\>}$ is an inner product
generating the norm ${\|\cdot\|_P}={\|P^\frac{1}{2}\cdot\|^\frac{1}{2}}.$
If $P$ is, in addition, surjective, so that it is invertible (i.e., strictly
positive), then the inner product space ${(\H,\<\,\cdot\,;\cdot\,\>_P)}$
becomes a Hilbert space whenever ${(\H,\<\,\cdot\,;\cdot\,\>)}$ is a Hilbert
space$.$

\vskip6pt
If $(A,B)$ is an $(m,P)$-biisometric pair and the positive $P$ is invertible,
then assuming that $L_{A^*}R_B$ is power bounded we get
$$
\triangle_{A^*,B}(P)=0
\quad\Longleftrightarrow\quad
A^*(PBP^{-1})=(P^{-1}A^*P)B=I,
$$
so that both $A$ and $B$ are left invertible$.$ Furthermore, if either of
$A^*$ and $B^*$ is injective, then both $A$ and $B$ are invertible: in
particular,
$$
{\N(A^*)}\ne\0
\quad\Longleftrightarrow\quad
{\N(B^*)}\ne\0,
$$
and $A$ is invertible if and only if $B$ is invertible$.$ The following
theorem shows that a stronger result than Theorem 5.1 is possible in the case
in which the operators $A$ and $B$ are power bounded$.$ But before that, some
notation and terminology is in order$.$ The numerical range $W(A)$ of an
operator $A$ is the set
$$
W(A)
=\big\{\lambda\in{\mathbb C}\!:\,\lambda=\<Ax\,;x\>,\;x\in\H,\;\|x\|=1\big\},
$$
and the numerical radius $w(A)$ of $A$ is
$$
w(A)={\rm sup} \{ |\lambda|\!:\,\lambda\in W(A)\}.
$$
The spectral radius $r(A)$ of $A$ is
$$
r(A)={\rm sup}\{|\lambda|\!:\,\lambda\in\sigma(A)\}
=\lim_{n\to\infty}{\|A^n\|^{\frac{1}{n}}}.
$$
$A$ is {\it normaloid}\/ if ${r(A)=\|A\|}$ (which implies $r(A)=w(A)=\|A\|$),
$A$ is {\it convexoid}\/ if the closure $\overline{W(A)}$ of the numerical
range of $A$ equals the convex hull ${{\rm conv}\kern.5pt\sigma(A)}$ of the
spectrum of $A$, and $A$ is {\it spectraloid}\/ if $r(A)=w(A)$
\cite[Problem 219]{Hal}$.$ It is well known that the classes consisting of
normaloid and convexoid operators are independent of each other, and that both
these classes are contained in the class of spectraloid operators.

\vskip6pt
We close the paper by showing that for an invertible positive $P$, a pair
$(A,B)$ of power bounded $(m,P)$-biisometric operators is such that either
$A$ and $B$ are both similar to the same unitary operator, or they satisfy the
conclusion of Theorem 5.1.

\begin{theorem}
\label{theorem2}
If\/ $\triangle^m_{A^*,B}(P)=0$ $($i.e.,\/ $(A,B)$ is an\/ $(m,P)$-biisometric
pair\/$)$ for some invertible positive\/ $($i.e., strictly positive\/$)$ $P$,
and\/ ${A,B}$ are power \hbox{bounded, then}
\begin{description}
\item{$\kern-4pt$\rm(a)}
either
\begin{description}
\item{$\kern-3pt$\rm(i)}
there exists a unitary operator\/ $U$ such that\/ $A$ and\/ $B$ are similar
to\/ $U$,
\end{description}
or
\vskip4pt
\begin{description}
\item{$\kern-4pt$\rm(ii)$\kern-2pt$}
$A$ and\/ $B$ satisfy the conclusions of Theorem 5.1\/ $($with some obvious
changes, since now\/ ${\R(P)=\H})$.
\end{description}
\vskip4pt
\item{$\kern-5pt$\rm(b)}
Furthermore, in case\/ {\rm(a-i)}, if\/ $A$ and\/ $A^{-1}\!$ $($or\/ $B$ and\/
$B^{-1}\kern-.5pt)$ are either \hbox{normaloid} or convexoid or spectraloid\/
$($all combinations are allowed\/$)$, then\/ $A$ is unitary\/
$($respectively,\/ $B$ is unitary$)$ and\/ ${B=P^{-1}AP}$ $($respectively,\/
$A=PBP^{-1})$.
\end{description}
\end{theorem}

\begin{proof}
The power boundedness of $A$ and $B$ implies (the power boundedness of
$L_{A^*}R_B$, and hence)
$$
\triangle_{A^*,B}(P)=0
\quad\Longleftrightarrow\quad
A^*=PB^{-1}P^{-1}
\quad\Longleftrightarrow\quad A=P^{-1}B^{*{-1}}P
$$
and, for all $x\in\H$ and positive integers $n$,
\begin{eqnarray*}
A^*PB=P
&\Longrightarrow&
A^{*n}PB^nP^{-1}=I                                                    \\
&\Longrightarrow&
\|x\|\le\|P^{-1}B^{*n}P\|\,\|A^nx\|\le M_1\|A^nx\|\le M_1M_2\|x\|
\end{eqnarray*}
for some positive scalars $M_1$ and $M_2.$ But then 
$$
{\smallfrac{1}{M_1}}\|x\|\le\|A^nx\|\le M_2\|x\|
\quad\hbox{for all}\quad
x\in\H
$$
implies the existence of an invertible operator $S$ and an isometry $V$ such
that $SA=VS$ \cite[Proposition 4.2]{KD3}$.$ Thus, since $A^*S^*SA=S^*S$,
there exists an invertible posi\-tive operator $P_1$, $P^2_1=S^*S$, and an
isometry $U$ such that $A^*P_1=P_1U^*$, or, $P_1A=UP_1$. Similarly, since
$$
\|x\|\le\|P^{-1}A^{*n}P\|\,\|B^nx\|\le M_{11}\|B^nx\|\le M_{12}\|x\|
$$
for some positive scalars $M_{11}, M_{12}$ and all $x\in\H$, there exists an
invertible positive operator $P_2$ and an isometry $V_2$ such that
$P_2B=V_2P_2$.

\vskip6pt\noi
(a)
The operators $A$ and $B$ being left invertible, if neither of $A^*$ and
$B^*$ is injective, then $\N(A^*)$ and $\N(B^*)$ are nonzero, and since
${\R(P)}=\H$, the argument of the proof of Theorem 5.1 goes through to prove
(a-ii)$.$ If one of $A^*$ and $B^*$ is injective, then both $A$ and $B$ are
invertible, $A=P^{-1}_1UP_1$ with the isometry $U$ being a unitary operator
and $B=P^{-1}A^{*{-1}}P=P^{-1}P_1UP^{-1}_1P=Q^{-1}UQ$ for an invertible
operator $Q.$ This proves (a-i).

\vskip6pt\noi
(b)
Assume now that $A$ and $A^{-1}$ are either normaloid or convexoid or
spectraloid. Then, since $\sigma(A^{\pm 1})$ is a subset of the boundary
${\partial{\mathbb D}}$ of the unit disc ${\mathbb D}$,
$$
r(A^{\pm 1})=w(A^{\pm 1})=1,
$$
\vskip-4pt\noi
and hence
$$
W(A^{\pm 1})\subseteq {\rm conv}\kern.5pt\sigma(A^{\pm 1}).
$$
\vskip4pt\noi
This \cite[Theorem 1]{St} implies that $A$ is a normal operator$.$ Since
$$
A^*P_1=P_1U^*
\quad\Longleftrightarrow\quad
AP_1=P_1U
$$
by the Putnam-Fuglede commutativity theorem (\cite[p.104]{Hal}),
$$
A^*P^2_1=P_1U^*P_1=P^2_1A^* \Longrightarrow A^*P_1=P_1A^*
\Longrightarrow P_1U^*=P_1A^*\Longleftrightarrow U=A.
$$
Trivially, $B^{-1}=P^{-1}A^*P=P^{-1}U^*P$ implies $B=P^{-1}UP.$ Since a
similar argument works for the case in which $B$ is normaloid or convexoid or
spectraloid, the proof is complete.
\end{proof}

\bibliographystyle{amsplain}

\end{document}